\documentclass[final,a4paper,reqno,11pt]{amsart}
\pdfoutput=1

\usepackage[final]{graphicx}
\usepackage[final]{hyperref}
\usepackage{amssymb, amstext, amsmath, amsthm, amsfonts, mathrsfs, enumerate, lmodern, ifdraft, stmaryrd}
\usepackage[noadjust]{cite}
\usepackage[utf8]{inputenc}
\usepackage[noabbrev]{cleveref}

\usepackage[activate={true,nocompatibility},final,tracking=true,kerning=true,spacing=true,factor=1100,stretch=10,shrink=10]{microtype}
\microtypecontext{spacing=nonfrench}

\newtheorem{theorem}{Theorem}[section]
\newtheorem{corollary}[theorem]{Corollary}
\newtheorem{lemma}[theorem]{Lemma}
\newtheorem{proposition}[theorem]{Proposition}

\newtheorem*{theorem*}{Theorem}
\newtheorem*{corollary*}{Corollary}
\newtheorem{introtheorem}{Theorem}

\Crefname{introcorollary}{corollary}{corollary}

\theoremstyle{definition}
\newtheorem{definition}[theorem]{Definition}
\newtheorem{example}[theorem]{Example}
\newtheorem{conjecture}[theorem]{Conjecture}
\newtheorem*{conjecture*}{Conjecture}
\newtheorem{construction}[theorem]{Construction}
\theoremstyle{remark}
\newtheorem{remark}[theorem]{Remark}
\newtheorem*{acknowledgment}{Acknowledgment}

\newcommand{\cC}{\mathcal{C}}

\newcommand{\cX}{\mathcal{X}}

\newcommand{\kk}{\Bbbk}
\newcommand{\NN}{\mathbb{N}}

\newcommand{\ZZ}{\mathbb{Z}}

\newcommand{\LL}{\mathbb{L}}

\newcommand{\add}{\mathsf{add}\hspace{.01in}}

\newcommand{\Db}{\mathsf{D}^{\mathrm{b}}}

\renewcommand{\mod}{\mathsf{mod}\hspace{.01in}}

\newcommand{\proj}{\mathsf{proj}\hspace{.01in}}

\newcommand{\End}{\operatorname{End}\nolimits}

\newcommand{\Ext}{\operatorname{Ext}\nolimits}
\newcommand{\Tor}{\operatorname{Tor}\nolimits}
\newcommand{\gl}{\operatorname{gldim}\nolimits}

\newcommand{\Hom}{\operatorname{Hom}\nolimits}
\newcommand{\id}{\operatorname{id}\nolimits}
\newcommand{\Id}{\operatorname{Id}\nolimits}

\renewcommand{\ker}{\operatorname{Ker}\nolimits}
\newcommand{\op}{\operatorname{op}\nolimits}
\newcommand{\pdim}{\operatorname{pdim}\nolimits}

\newcommand{\rad}{\operatorname{rad}\nolimits}
\newcommand{\soc}{\operatorname{soc}\nolimits}
\renewcommand{\top}{\operatorname{top}\nolimits}

\renewcommand{\subset}{\subseteq}

\newcommand{\includepdf}[1]{
   \begin{center}
      \includegraphics{#1}
   \end{center}
}
\newcommand{\includeformula}[2][]{
   \vcenter{\hbox{\includegraphics[#1]{#2}}}
}

\newcommand{\set}[2]{\left \{ \, {#1} \, \middle \vert \, {#2} \, \right \}}
\newcommand{\s}{\mbox{}}

\newcommand*{\doi}[1]{\href{http://dx.doi.org/#1}{doi:\s\ #1}}

\begin{document}
   \title[Minimal projective resolutions over idempotent subrings]{A method for constructing minimal projective resolutions over idempotent subrings}

   \author{Carlo Klapproth}\address{Department of Mathematics, Aarhus University, 8000 Aarhus C, Denmark} 
   \email{carlo.klapproth@math.au.dk}
    

   \begin{abstract}
      We show how to obtain minimal projective resolutions of finitely generated modules over an idempotent subring $\Gamma_e := (1-e)R(1-e)$ of a semiperfect noetherian basic ring $R$ by a construction inside $\mod R$.
      This is then applied to investigate homological properties of idempotent subrings $\Gamma_e$ under the assumption of $R/\langle 1-e\rangle$ being a right artinian ring.
      In particular, we prove the conjecture by Ingalls and Paquette that a simple module $S_e := eR /\rad eR$ with $\Ext_R^1(S_e,S_e) = 0$ is self-orthogonal, that is $\Ext^k_R(S_e,S_e)$ vanishes for all $k \geq 1$, whenever $\gl R$ and $\pdim eR(1-e)_{\Gamma_e}$ are finite.
      Indeed, a slightly more general result is established, which applies to sandwiched idempotent subrings:
      Suppose $e \in R$ is an idempotent such that all idempotent subrings $\Gamma$ sandwiched between $\Gamma_e$ and $R$, that is $\Gamma_e \subset \Gamma \subset R$, have finite global dimension.
      Then the simple summands of $S_e$ can be numbered $S_1, \dots, S_n$ such that $\Ext_R^k(S_i, S_j) = 0$ for $1 \leq j \leq i \leq n$ and all $k > 0$.
   \end{abstract}

   \maketitle

   \section*{Introduction} 
   This paper combines methods used in \cite{ingalls_homological_2015}, \cite{ingalls_homological_2017} and \cite{bravo_idempotent_2019}.

   For convenience, let $\kk$ \emph{always} be an algebraically closed field unless mentioned otherwise.
   Imagine the following situation:
   Given a semiperfect noetherian basic ring $R$ and an idempotent $e \in R$. 
   There are three questions that naturally arise:
   \begin{enumerate}
      \item What do desirable homological properties of $R$ imply for the homological behaviour of $\Gamma_e := (1-e)R(1-e)$ and which assumptions are needed for such implications? \label{enum:intro_first_question}
      \item What do desirable homological properties of $\Gamma_e$ imply for the homological behaviour of $R$ and which assumptions are needed for such implications? \label{enum:intro_second_question}
      \item Which set of assumptions is optimal in (\ref{enum:intro_first_question}) and (\ref{enum:intro_second_question})? \label{enum:intro_third_question}
   \end{enumerate}
   We will discuss this with respect to global and projective dimensions and will give a rough overview of Ingalls's, Paquette's and Bravo's work related to our paper:

   With regard to semiperfect noetherian rings, these questions have been tackled in \cite{ingalls_homological_2015} and \cite{ingalls_homological_2017}: 
   The semisimple module $S_e := eR/\rad eR$ and the quotient ring $R/\langle 1-e\rangle$ of $R$ by the ideal generated by the idempotent $1-e$ play an important role for question (\ref{enum:intro_second_question}) as $\Gamma_e$ cannot see $S_e$.
   Indeed for question (\ref{enum:intro_second_question}) the best result one could expect was established in \cite[Proposition 4.3]{ingalls_homological_2017}:
   If $\Ext_R^k(S_e, S_e)$ vanishes for all but finitely many $k \in \NN$ and $\Gamma_e$ has finite global dimension, then $R$ does too.

   On the other hand question (\ref{enum:intro_first_question}) is more difficult. 
   A ring can have homologically more complex idempotent subrings.
   Indeed, Auslander has shown in \cite{auslander1971representation} that every Artin algebra is an idempotent subalgebra of an Artin algebra of finite global dimension.
   
   Investigating question (\ref{enum:intro_first_question}), Ingalls and Paquette consider the (graded) Yoneda ring $Y_e := \Ext_R^\ast(S_e, S_e)$ of $S_e$ in \cite{ingalls_homological_2017}.
   This ring links $\gl R$ and $\gl \Gamma_e$ together:
   If $Y_e$ and $R$ have finite global then so does $\Gamma_e$.
   However, the assumption that $Y_e$ is of finite global dimension is stronger than necessary, as \cite[page 314]{ingalls_homological_2017} states.
   This matter is further discussed in \cite{ingalls_homological_2017}, but not finally resolved yet.
   With the framework we developed we are able to prove that $Y_e$ must indeed be of finite global dimension if every idempotent subring between $\Gamma_e$ and $R$ has finite global dimension:
   \begin{introtheorem}[Cf.\s\ \Cref{corollary:triangular}] \label{introcorollary:triangular}
      Let $R$ be a semiperfect noetherian basic ring of finite global dimension and $e \in R$ be an idempotent such that $R/\langle 1-e\rangle$ is right artinian.
      Then all idempotent subrings sandwiched between $\Gamma_e$ and $R$ have finite global dimension if and only if $Y_e$ is a directed artinian ring (see \Cref{def:triangular}).
   \end{introtheorem}

   Questions (\ref{enum:intro_first_question}) and (\ref{enum:intro_third_question}) also have been fully answered by Ingalls and Paquette for $e$ primitive over finite dimensional $\kk$-algebras $R$:
   If $S_e$ is self-orthogonal, i.e.\s\ $\Ext_R^k(S_e,S_e) =0$ for all $k \geq 1$, and $R$ has finite global dimension then $\Gamma_e$ has finite global dimension as shown in \cite[Proposition 4.4]{ingalls_homological_2015}.
   On the other hand, $\gl R < \infty$ and $\gl \Gamma_e < \infty$ imply that $S_e$ is self-orthogonal, by \cite[Theorem 6.5]{ingalls_homological_2015} using the strong no loops conjecture shown in \cite{igusa_proof_2011}.
   Indeed \cite[Theorem 6.5]{ingalls_homological_2015} does not require $R$, but only $R/\rad R$, to be finite dimensional over $\kk$.
   Ingalls and Paquette conjectured that this theorem generalizes to semiperfect noetherian $\kk$-algebras:
   
   \begin{conjecture*}[{\cite[Conjecture 4.13]{ingalls_homological_2015}}, cf.\s\ \Cref{conj:ingalls}] \label{introconj:ingalls}
     Let $R$ be a noetherian $\kk$-algebra which is either semiperfect or positively graded. 
     Assume that $e \in R$ is a primitive idempotent (and of degree zero if $R$ is positively graded) with $\Ext_R^1(S_e,S_e) = 0$.
     If both $\pdim_R S_e$ and $\pdim_{\Gamma_e} eR(1-e)$ are finite, then $S_e$ is self-orthogonal.
   \end{conjecture*}

   With respect to semiperfect noetherian rings we can give a proof for this in \Cref{proof:ingalls} on \cpageref{proof:ingalls}. 
   Indeed it will be a corollary of a more general statement:
   \begin{introtheorem}[Cf.\s\ \Cref{corollary:generalization}] \label{introcorollary:generalization}
      Let $R$ be a semiperfect noetherian basic ring, $e \in R$ be a primitive idempotent and $R/\langle 1-e\rangle$ be artinian.
      If $\Tor_k^{\Gamma_e}(eR(1-e), (1-e)Re)$ and $\Ext_{R}^k(S_e,S_e)$ vanish for all but finitely many $k \in \NN$, then $S_e$ is self-orthogonal.
   \end{introtheorem}

   \Cref{introcorollary:triangular} and \Cref{introcorollary:generalization} are consequences of the main results of this paper which are
   \begin{enumerate}
      \item \emph{\Cref{cons:the_construction}} which allows us to calculate a complex in $\mod R$ which maps to a minimal projective resolution under the functor $- \otimes_R R(1-e)$ and
      \item the introduction of \emph{\Cref{def:m-property}}. 
   \end{enumerate}

   \section{Preliminaries and notations} \label{sec:preliminaries}

   For morphisms $X \xrightarrow{f} X'$ and $X' \xrightarrow{g} X''$ we denote the composite $g \circ f \colon X \xrightarrow{f} X' \xrightarrow{g} X''$ also by $gf$ and for functors $F \colon \cC \to \cC'$ and $G \colon \cC' \to \cC''$ we also denote their composite $G \circ F \colon \cC \to \cC''$ by $GF$.
   All modules in this paper are seen as right modules.
   For a subcategory $\cX \subset \cC$ we will define the left perpendicular subcategory of $\cX$ in $\cC$ to be the full subcategory ${}^\perp \cX := \{ C \in \cC \,|\, \forall X \in \cX \colon \Hom_{\cC}(C, X) = 0 \}$.

   For convenience, we introduce a slightly modified notation to that of \cite{ingalls_homological_2015}, \cite{ingalls_homological_2017} and \cite{bravo_idempotent_2019}.
   Throughout this paper let $R$ be a semiperfect noetherian ring.
   We want to emphasize that throughout this paper \emph{noetherian} means left and right noetherian.
   For details on semiperfect (noetherian) rings we refer the reader to \cite[$\S$27]{anderson_rings_1992}.
   However, we want to remind the reader that $\proj R$ and $\add R / \rad R$ are Krull--Schmidt categories by \cite[Theorem 12.6]{anderson_rings_1992} and that every semiperfect ring is Morita equivalent to a basic semiperfect ring by \cite[Proposition 27.14]{anderson_rings_1992}. 
   As being noetherian is preserved by Morita equivalence we can assume that \emph{$R$ is basic} by the latter.
   Notice, for a non-basic ring $R$ some theorems need a slight reformulation.
   Let $e \in R$ always be an arbitrary idempotent.
   We denote by $\Gamma_{e}$ the idempotent subring $(1-e)R(1-e)$ obtained from $R$ by deleting $e$, by $F_{e}$ the functor\footnote{This functor indeed maps finitely generated modules to finitely generated modules (see \Cref{rmk:functor_preserves_fg_modules}).} $- \otimes_R R(1-e) \colon \mod R \to \mod \Gamma_{e}$ and by $S_{e}$ the semisimple module $eR / \rad eR$ belonging to the idempotent $e$.
   The latter semisimple module is in the kernel of $F_e$.
   By $Y_{e}$ we will denote the (graded) Yoneda ring $\Ext_R^\ast(S_{e},S_{e})$ with the Yoneda product as multiplication.

   As ${}_R R(1-e)$ is projective, $F_{e}$ is an exact functor. 
   Let $\langle 1-e \rangle$ be the (two sided) ideal of $R$ generated by $1-e$. 
   Recall the following:

   \begin{lemma} \label{lemma:serre}
      The category $\mod R/\langle 1-e\rangle$ is a full subcategory of $\mod R$ via restriction of scalars using the canonical ring epimorphism $R \to R/\langle 1-e\rangle$. 
      Further $\mod R/\langle 1-e\rangle$ is the kernel of $F_e$.
   \end{lemma}
   \begin{proof}
      By restriction of scalars every $R/\langle 1-e\rangle$-module becomes an $R$-module without changing the underlying set.
      This way $\mod R/\langle 1-e\rangle$ becomes a subcategory of $\mod R$.
      Because $R \to R /\langle 1-e\rangle$ is surjective, an $R$-linear mapping between $R/\langle 1-e\rangle$-modules is $R/\langle 1-e\rangle$-linear.
      This establishes that $\mod R/\langle 1-e\rangle \subset \mod R$ is full.

      Now if $(1-e)$ operates trivially on an $R$-module $M$ then the ideal $\langle 1-e\rangle$ is contained in the annihilator of $M$ and hence the morphism of rings $R \to \End_{\ZZ}(M)^{\op},\, r \mapsto (m \mapsto mr)$ factors through $R/\langle 1-e\rangle$.
      This shows that $M$ is an $R/\langle 1-e\rangle$-module in this case.
      Therefore, we can describe $\mod R/\langle 1-e\rangle$ as the subcategory of $\mod R$ annihilated by $(1-e)$.
      Because $F_{e}$ is actually just right multiplication with $(1-e)$ the second part of the lemma follows.
   \end{proof}
   
   The subcategory $\mod R/\langle 1-e\rangle \subset \mod R$ has more favourable properties.
   Recall the following definition:
   \begin{definition}
      Let $\cC$ be a category and $\cX \subset \cC$ be a full subcategory.
      A morphism $f \colon C \to X$ from an object $C \in \cC$ to an object $X \in \cX$ is called
      \begin{enumerate}
         \item \emph{$\cX$-preenvelope} if $\Hom_\cC(f, X') \colon \Hom_\cC(X,X') \to \Hom_\cC(C,X')$ is surjective for every $X' \in \cX$,
         \item \emph{$\cX$-envelope} if it is an $\cX$-preenvelope and every $g \in \End_\cC(X)$ with $gf = f$ is an isomorphism and
         \item \emph{strong $\cX$-envelope} if $\Hom_\cC(f, X') \colon \Hom_\cC(X,X') \to \Hom_\cC(C,X')$ is bijective for every $X' \in \cX$.
      \end{enumerate}
      We say $\cX$ is \emph{preenveloping}, \emph{enveloping} respectively \emph{strongly enveloping} in $\cC$ if every object $C \in \cC$ admits an $\cX$-preenvelope, $\cX$-envelope respectively strong $\cX$-envelope.
   \end{definition}
   Every strong $\cX$-envelope is an $\cX$-envelope as $\Hom_\cC(f,X) \colon \Hom_\cC(X,X) \to \Hom_\cC(C,X)$ being injective implies that the only $g \in \End_\cC(X)$ with $gf = f$ is $g = \id_X$.

   Let $L_e := - \otimes R/\langle 1-e \rangle \colon \mod R \to \mod R /\langle 1-e \rangle$ and let $I_e \colon \mod R/\langle 1-e \rangle \to \mod R$ be the canonical inclusion functor.
   As for example mentioned in \cite[Example 3.3]{green_reduction_2018}, there is a recollement
   \includepdf{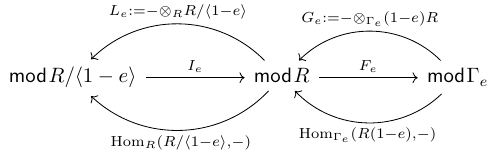}
   of abelian categories.
   We refer the reader to \cite[Proposition 3.2]{green_reduction_2018} and \cite{psaroudakis_homological_2014} for further details on this recollement. 
   In particular, recall the following:
   
   \begin{lemma}
      The category $\mod R/\langle 1-e \rangle$ is an enveloping Serre subcategory of $\mod R$.
   \end{lemma}
   \begin{proof}
      As $\mod R/\langle 1-e\rangle$ is the kernel of the exact functor $F_e$ it is a Serre subcategory of $\mod R$.
      We have an adjoint functor pair $L_e \dashv I_e$ with unit $\eta \colon \Id_{\mod R} \to I_e L_e$.
      Notice, $I_e$ is the inclusion of $\mod R/\langle 1-e \rangle$ into $\mod R$, so we have $I_e(X) = X$ for $X \in \mod R/\langle 1 -e \rangle$ and $I_e$ is fully faithful.
      Therefore, \cite[Proposition 2.4.1]{krause_localization_2009} shows that $I_e L_e (\eta_M)$ is invertible for $M \in \mod R$. 
      For any $X \in \mod R/\langle1-e\rangle$ and any $M \in \mod R$ the commutative diagram
      \includepdf{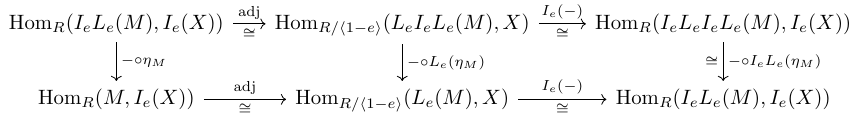}
      now shows that $\Hom_R(\eta_M, I_e(X))\colon \Hom_R(I_e L_e (M), I_e(X)) \to \Hom_R(M, I_e(X))$ is an isomorphism.
      So, the unit $\eta_M$ defines a strong $\mod R/\langle 1-e\rangle$-envelope for $M \in \mod R$.
   \end{proof}

   This implies that the simple modules of $\mod R/\langle 1-e\rangle$ are precisely the simple modules of $\mod R$ lying in the kernel of $F_e$ and that $I_e$ preserves and reflects kernels, cokernels and short exact sequences.
   In particular, $R/\langle 1-e \rangle$ has a finite composition series as an $R/\langle 1-e\rangle$-module if and only if it has a finite composition series as an $R$-module.
   Therefore, $R/\langle 1-e\rangle$ is a right artinian ring if and only if $R/\langle 1-e\rangle_R$ is a right artinian $R$-module. 
   It is a semisimple ring if and only if $\Ext_R^1(S_e,S_e) = 0$ as the inclusion of a Serre subcategory induces an isomorphism $\Ext_R^1(S_e,S_e) \cong \Ext_{R/\langle 1-e\rangle}^1(S_e,S_e)$ and because $S_e = (R/\langle 1-e\rangle)/(\rad R/\langle 1-e\rangle)$ as an $R/\langle 1-e\rangle$-module.

   In general envelopes with respect to subobject closed subcategories of abelian categories are epic. 
   In our case this is trivially satisfied as the $\mod R /\langle 1-e\rangle$-envelope $\eta_M$ of $M \in \mod R$ is the canonical quotient map $M \to M/MI$ modulo the submodule $MI \subset M$, where ${I = \langle 1-e \rangle}$.
   So, the following is just the canonical sequence $0\to MI\to M \to M/MI \to 0$:

   \begin{lemma} \label{lemma:r1-eenvelope}
      For $M \in \mod R$ the $R/\langle 1-e\rangle$-envelope $\eta_M \colon M \to I_e L_e(M)$ of $M$ induces a short exact sequence
      \includepdf{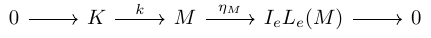}
      with $I_e L_e (M) \in \mod R/\langle 1-e\rangle$ and $K$ in ${}^\perp\mod R/\langle 1-e\rangle$.
    \end{lemma}
    \begin{proof}
      The morphism $\eta_M$ is epic as $\mod R/\langle 1-e\rangle \subset \mod R$ is a Serre subcategory.
      For the same reason the subcategory $\mod R/\langle 1-e\rangle \subset \mod R$ is extension closed.
      Now notice $\mod R/\langle 1-e\rangle \subset \mod R \subset \Db(\mod R)$ is also a full and extension closed subcategory.
      By the dual of the triangulated version of Wakamatsu's lemma found in \cite[Lemma 2.1]{jorgensen_auslander-reiten_2006} we obtain that the kernel $k \colon K \to M$ of $\eta_M$ satisfies $\Hom_{\smash{\Db(\mod R)}}(K, X') = 0$ for all $X' \in \mod R/\langle 1-e\rangle$ which shows $K \in {}^{\perp}\mod R/\langle 1-e\rangle$.
   \end{proof}
  
   \begin{remark} \label{remark:projcover}
   In the above, as $R$ is semiperfect the projective cover of the finitely generated module $K$ can be calculated as follows: 
   Take a projective cover $p' \colon P \to \top K$ and factor it through the canonical epimorphism $t \colon K \to \top K$, say $tp = p'$ for some $p \colon P \to K$.
   By Nakayama's lemma $p$ is an epimorphism and because $p' \colon P \to \top K$ had radical kernel $p$ does so too.
   Hence, $p$ is a projective cover of $K$.
   However, $K \in {}^\perp\mod R/\langle 1-e\rangle$ means that no simple summand of $S_e$ appears in the top of $K$.
   Hence, the projective cover of $K$ lies in $\add R(1-e)$.
   \end{remark}

   Notice that the terms of a minimal projective resolution of an $R$-module $M$ determine the groups $\Ext_R^k(M,S_e)$ for all $k \in \NN$ completely:
   \begin{remark} \label{remark:determine}
      Let $M \in \mod R$ and $\cdots \to P_2 \to P_1 \to P_0 =: P_\ast$ be a minimal projective resolution of $M$.
      Now $\Ext_R^k(M,S_e)$ is isomorphic to the degree $k$ homology of the image of $P_\ast$ under $\Hom_R(-,S_e)$.
      However, applying $\Hom_R(-, S_e)$ to a morphism with radical image yields the zero morphism, because $S_e$ is semisimple.
      As all morphisms in $P_\ast$ have radical image we have $\Ext_R^k(M, S_e) \cong \Hom_R(P_k, S_e)$.
      In particular, $P_k \in \add (1-e)R$ if and only if $\Ext_R^k(M, S_e) = 0$.
   \end{remark}
   Hence, the following lemma can be used to show that $F_e$ preserves minimal projective resolutions of $R$-modules $M$ with $\Ext_R^\ast(M,S_e) = 0$:

   \begin{lemma} \label{lemma:equivalence}
      The restriction $F_e \colon \add (1-e)R \to \proj \Gamma_e$ is an equivalence and preserves and reflects radical morphisms.
   \end{lemma}
   \begin{proof}
      The inverse of the restriction of $F_e$ is given by $- \otimes_{\Gamma_e} (1-e)R \colon \proj \Gamma_e \to \add (1-e)R$.
      A morphism between projective modules over a semiperfect noetherian ring is radical if and only if it does not induce an isomorphism between a non-trivial direct summand in its domain to a direct summand in its codomain.
      This property is clearly preserved and therefore reflected by additive equivalences between additive subcategories. 
   \end{proof}

   \section{The Construction} \label{sec:the_construction}
   With the preliminary work done, we can now formulate our main result:
   How to calculate minimal projective resolutions over an idempotent subring $\Gamma_e$ by a construction in $\mod R$.

   The main idea of this \Cref{cons:the_construction} is as follows: 
   The functor $F_e$ preserves epis and for a projective $R$-module $P$ the image $F_e(P)$ is a projective $\Gamma_e$-module if $P \in \add (1-e)R$.
   Projective covers over a semiperfect noetherian ring are uniquely determined by the top of the module covered.
   Hence, $F_e$ preserves projective covers of $R$-modules not having a summand of $S_e$ in their top.
   One can use $R/\langle 1-e\rangle$ approximations to remove precisely the unwanted summands of $S_e$ from the top of modules.
   Calculating projective resolutions can then be done iteratively by calculating projective covers.

   \begin{construction} \label{cons:the_construction}
      Let $\Omega_0$ be a finitely generated $R$-module.
      Define inductively
      \begin{enumerate}
         \item $f_i \colon \Omega_i \to X_i$ as the $\mod R/\langle 1-e\rangle$-envelope of $\Omega_i$ defined in \Cref{lemma:r1-eenvelope}, 
         \item $k_i \colon K_i \to \Omega_i$ as the kernel of $f_i$ as in \Cref{lemma:r1-eenvelope}, 
         \item $g_i \colon P_i \to K_i$ as a projective cover of $K_i$ and 
         \item $h_i \colon \Omega_{i+1} \to P_i$ as the kernel of $g_i$ 
      \end{enumerate}
      for all $i \in \NN$.
      This way the infinite tree in \Cref{fig:infinite_tree}
      \begin{figure}[ht]
         \includepdf{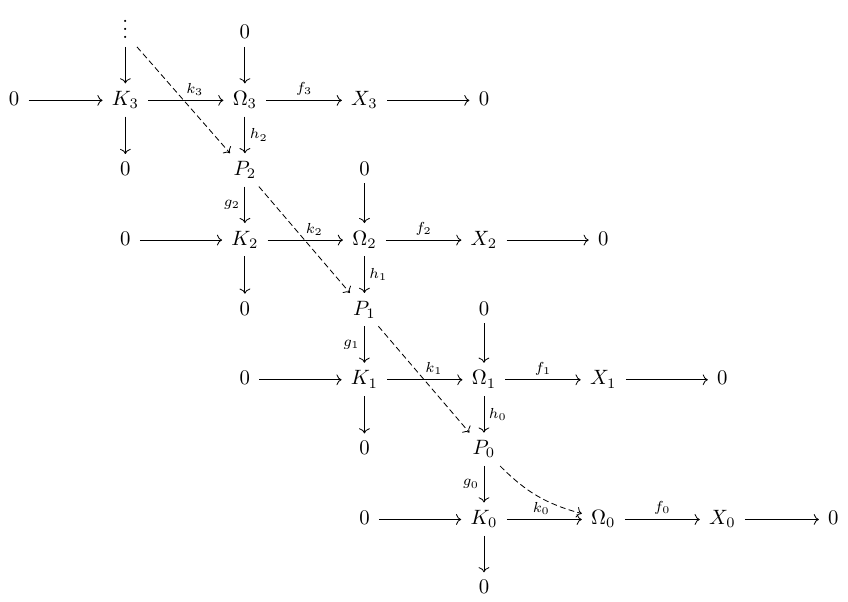}
         \caption{The infinite tree obtained from \Cref{cons:the_construction}}
         \label{fig:infinite_tree}
      \end{figure}
      is obtained, where all rows and columns are exact sequences and the dashed arrows are defined by composition.
      Notice, the horizontal short exact sequences are precisely the short exact sequences constructed in \Cref{lemma:r1-eenvelope}.
      Moreover, $P_i \in \add (1-e)R$ for $i \in \NN$ by \Cref{remark:projcover} as $K_k$ is in ${}^\perp \mod R/\langle 1-e\rangle$ by \Cref{lemma:r1-eenvelope}.
      Applying $F_e$ to the dashed arrows yields an augmented minimal projective resolution of $F_e(\Omega_0)$ as shown in \Cref{thm:the_construction_works}.
      As $F_e(X_\ast)$ vanishes, $F_e(\Omega_i) \cong F_e(K_i)$ is the $i$-syzygy of $F_e(\Omega_0)$.
      Moreover, $X_i$ is the degree $i$ homology of $F_e(\Omega_0) \otimes_{\smash{\Gamma_e}}^{\smash{\LL}}(1-e)R$ for $i > 0$, where $-\otimes_{\Gamma_e}^{\smash{\LL}} (1-e)R$ is the left derived functor of $- \otimes (1-e)R \colon \mod \Gamma_e \to \mod R$.
      This is shown in \Cref{prop:homology}.
   \end{construction}
  
   \Cref{cons:the_construction} is a generalization of the methods used in \cite[Proposition 4.9]{ingalls_homological_2015} dropping the assumption $\Ext_R^1(S_e,S_e) = 0$.
   This method also took inspiration from the method shown in \cite[Lemma 2.3]{bravo_idempotent_2019}.

   This construction is particularly helpful for homological questions, as we will see in \Cref{sec:applications}, but we can also apply this to ring theoretic questions:

   \begin{remark} \label{rmk:functor_preserves_fg_modules}
      As $F_e(P)$ is a finitely generated projective $\Gamma_e$-module for $P \in \add(1-e)R$ the first step of \Cref{cons:the_construction} shows that $F_e$ maps finitely generated $R$-modules to finitely generated $\Gamma_e$-modules.
   \end{remark}

   The following \namecref{thm:the_construction_works} uses the notation from \Cref{cons:the_construction}:
   \begin{theorem} \label{thm:the_construction_works}
      Let $\Omega_0 \in \mod R$.
      Then the augmented complex $(C_\ast,d_\ast,\varepsilon)$ with objects $C_i := F_e(P_i)$, differentials $d_i := F_e(h_i k_{i+1} g_{i+1})  \colon F_e(P_{i+1}) \to F_e(P_i)$ in degrees $i \in \NN$ and augmentation $\varepsilon := k_0 g_0 \colon F_e(P_0) \to F_e(\Omega_0) $ is a minimal projective resolution of $F_e(\Omega_0)$.
   \end{theorem}
   \begin{proof}
      Notice, the considered complex $(C_\ast,d_\ast,\varepsilon)$ is the image of the dashed complex shown in \Cref{fig:infinite_tree} under $F_e$.
      As $X_i \in \mod R/\langle 1-e\rangle = \ker F_e$ and because $F_e$ is exact, the morphism $F_e(k_i)$ is an isomorphism and $F_e(g_i)$ is epic with kernel $F_e(h_i)$ for $i \in \NN$.
      We have $K_i \in {}^\perp\mod R/\langle 1-e\rangle$ by \Cref{lemma:r1-eenvelope} and this implies $P_i \in \add(1-e)R$ for $i \in \NN$ by \Cref{remark:projcover}.
      Hence, $F_e(P_i)$ is in $\proj \Gamma_e$ by \Cref{lemma:equivalence}.
      This shows that $(C_\ast,d_\ast,\varepsilon)$ is an augmented projective resolution of $F_e(\Omega_0)$.
      As $F_e$ preserves radical maps between projective modules in $\add (1-e)R$ by \Cref{lemma:equivalence} all differentials of $(C_\ast,d_\ast)$ are radical and hence have radical images.
      Hence, by exactness, $\varepsilon$ has radical kernel and therefore $(C_\ast,d_\ast,\varepsilon)$ is minimal.
   \end{proof}
   
   \Cref{cons:the_construction} and \Cref{thm:the_construction_works} can be extended to abelian categories fulfilling additional assumptions.
   For example, since injective envelopes exist over any ring, a similar result for injective resolutions over arbitrary rings $R$ can be obtained by appropriately modifying the proof given, at the expense of further technicalities. 
   
   Let $G_e := - \otimes_{\Gamma_e} (1-e)R \colon \mod \Gamma_e \to \mod R$. 
   This functor is right exact but in general not exact. 
   However, we can easily calculate the homology of its left derived functor using \Cref{cons:the_construction}.

   \begin{proposition} \label{prop:homology}
      Let $\Omega_0 \in \mod R$.
      Then $\Tor_i^{\Gamma_e}(F_e(\Omega_0), (1-e)R) \cong X_{i+1}$ for $i \geq 1$.
   \end{proposition}
   \begin{proof}
      To calculate $\Tor_i^{\Gamma_e}(F_e(\Omega_0), (1-e)R)$ in degree $i \geq 1$ we have to apply $G_e$ term wise to a projective resolution of $F_e(\Omega_0)$ and calculate the degree $i$ homology of the resulting complex. 
      By \Cref{thm:the_construction_works} and because $G_e F_e$ is naturally isomorphic to the identity on $\add (1-e)R$ this homology is just the homology of the dashed complex in \Cref{fig:infinite_tree}.
      As $k_i$ and $h_{i-1}$ are injective we have $\ker(h_{i-1} k_i g_i) = \ker(g_i) = \Omega_{i+1}$ for $i \geq 1$.
      Moreover, as $g_{i+1}$ is surjective for $i \geq 1$ we obtain $\Tor_i^{\Gamma_e}(F_e(\Omega_0), (1-e)R) \cong \Omega_{i+1}/k_{i+1}(K_{i+1}) \cong X_{i+1}$.
   \end{proof}

   \Cref{prop:homology} is helpful to get a termination condition for \Cref{cons:the_construction} when one has some knowledge about projective dimensions $\Gamma_e$-modules.
   For example if one knows that $\Gamma_e$ is of global dimension $n \in \NN$, then in \Cref{cons:the_construction} only the first $n+1$ iterations need to be done.

   \section{Applications} \label{sec:applications}
   We want to use \Cref{cons:the_construction} to prove the following conjecture by Ingalls and Paquette for semiperfect noetherian rings (a proof is given on \cpageref{proof:ingalls}):
    
   \begin{conjecture}[{\cite[Conjecture 4.13]{ingalls_homological_2015}}] \label{conj:ingalls}
     Let $R$ be a noetherian $\kk$-algebra which is either semiperfect or positively graded. 
     Assume that $e\in R$ is a primitive idempotent (and of degree zero if $R$ is positively graded) with $\Ext_R^1(S_e,S_e) = 0$.
     If both $\pdim_R S_e$ and $\pdim_{\Gamma_e} eR(1-e)$ are finite, then $S_e$ is self-orthogonal.
   \end{conjecture}

   But we also want to extend this conjecture to non-primitive idempotents to prove \Cref{thm:sandwiched_subrings}. 
   For convenience we introduce the following notation:
   Assume that $M$ and $M'$ are $R$-modules.
   Then we denote by 
   \begin{equation} m(M,M') := \sup \set{k \in \NN}{\Ext_R^k(M,M') \neq 0} \label{eqn:m} \end{equation}
   the supremum of all degrees in which $M$ has extensions by $M'$, where the supremum is taken in $\NN \cup \{\pm \infty \}$.
   In particular $m(M,M') = -\infty$ means that $M$ has no extensions by $M'$ and no morphisms to $M'$ and $m(M,M') = \infty$ means that $M$ has extensions by $M'$ in infinitely many degrees.
   We observe that $m(M,M')$, for simple modules $M$ and $M'$, cannot be $0$ unless $M \cong M'$.

   We introduce the following property, which is a slight generalization of the Yoneda ring of $M$ being of uniformly graded Loewy length, as for example considered in \cite{bravo_idempotent_2019}:

   \begin{definition} \label{def:m-property}
      Let $\cX \subset \mod R$ be a subcategory and $m$ as in \cref{eqn:m}.
      \begin{enumerate}
         \item For $X \in \cX$ we say that $X' \in \cX$ is an \emph{$m$-dual} of $X$ with respect to $\cX$ if $X'$ is non-zero and the inequality $m(X,X') \geq m(X'', X')$ holds for all $X'' \in \cX$.
         \item We say that $\cX$ has the \emph{$m$-property} if the restriction $m(\left.-\right\vert_\cX,\left.-\right\vert_\cX)$ is a bounded above function and every non-zero $X \in \cX$ has an $m$-dual with respect to $\cX$.
      \end{enumerate} 
      Similarly, we say $M \in \mod R$ has the \emph{$m$-property} if $\add M$ does so. 
   \end{definition}
   
   It is not difficult to construct a finite dimensional $\kk$-algebra $R$ and a semisimple module $S_e$ having the $m$-property such that $Y_e = \Ext^\ast_R(S_e, S_e)$ is not of uniform graded Loewy length.
   But whether the $m$-property is of any use, except for the proof of \Cref{thm:sandwiched_subrings}, has yet to be determined.

   \begin{remark} \label{remark:mproperty}
   Let $M := M_1 \oplus \dots \oplus M_n$ be a direct sum of indecomposable $R$-modules with local endomorphism ring.
   Then $\add M$ is Krull--Schmidt by \cite[Theorem 12.6]{anderson_rings_1992} and to check that $M$ has the $m$-property we first need to verify that $M$ has self-extensions in only finitely many degrees.
   If this is the case then $M$ having the $m$-property is equivalent to the statement that in each row of \Cref{tab:m-property},
   \begin{table}[ht]
      \caption{The values of $m(-,-)$ on the summands of $M_1 \oplus \dots \oplus M_n$}
      \begin{tabular}{c|cccc}
         $m(-,-)$ & $M_{1}$           & $M_{2}$           & $\dots$  & $M_{n}$           \\
         \hline
         $M_{1}$  & $m(M_{1}, M_{1})$ & $m(M_{1}, M_{2})$ & $\dots$  & $m(M_{1}, M_{n})$ \\
         $M_{2}$  & $m(M_{2}, M_{1})$ & $m(M_{2}, M_{2})$ & $\dots$  & $m(M_{2}, M_{n})$ \\
         $\vdots$ & $\vdots$          & $\vdots$          & $\ddots$ & $\vdots$          \\
         $M_{n}$  & $m(M_{n}, M_{1})$ & $m(M_{n}, M_{2})$ & $\dots$  & $m(M_{n}, M_{n})$
      \end{tabular}
      \label{tab:m-property}
   \end{table}
   there is a value which is maximal within its column.
   \end{remark}

   Notice that \cite[Lemma 2.2]{bravo_idempotent_2019} shows that $S_e$ is an $m$-dual to each non-zero $M \in \add S_e$ with respect to $\mod R/\langle 1-e\rangle$ if $R/\langle 1-e \rangle$ is artinian and $Y_e$ is of uniformly graded Loewy length.

   \begin{lemma} \label{lem:extend_m-property}
   Suppose $S_e$ has the $m$-property and that $R/\langle 1-e\rangle$ is right artinian.
   Then 
   \begin{enumerate}
      \item every $X \in \mod R/\langle 1-e\rangle$ has an $m$-dual in $\add S_e \subset \mod R/\langle 1-e\rangle$ with respect to $\mod R/\langle 1-e\rangle$ and \label{enum:lemma1}
      \item if $Y_e$ is of uniformly graded Loewy length then the semisimple module $S_e$ is an $m$-dual to each $X \in \mod R/\langle 1-e\rangle$ with respect to $\mod R/\langle 1-e\rangle$. \label{enum:lemma2}
   \end{enumerate}
   Moreover, $\mod R/\langle 1-e\rangle$ has the $m$-property.
   \end{lemma}
   \begin{proof}
      If $R/\langle 1-e\rangle$ is right artinian then each object of $\mod R/\langle 1-e\rangle$ has a finite composition series (in $\mod R$) with composition factors in $\add S_e$.
      
      First we show that for all $X', X'' \in \mod R /\langle 1-e\rangle$ the inequality
      \begin{equation}
         m(X'', X') \leq m(S_e, X') \label{eqn:Sedominates}
      \end{equation}
      holds.
      Let $X' \in \mod R/\langle 1-e\rangle$ and define $m := m(S_e, X')$.
      If $m = \infty$ the claim clearly holds, so we may assume $m < \infty$.
      If $0 = X_0'' \subset X_1'' \subset \dots \subset X_{i}'' = X''$ is a composition series of $X'' \in \mod R/\langle 1-e\rangle$ then the exact sequences
      \includepdf{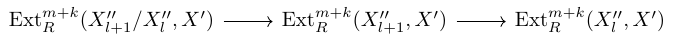}
      for $k \in \NN$ and $l \in \{0,\dots,i-1\}$ show $m(X_{l+1}'',X') \leq m$ inductively for $l = 0, \dots, i-1$ since $X_{l+1}'' / X_{l}'' \in \add S_e$ for $l = 0, \dots, i-1$.
      Hence, \cref{eqn:Sedominates} holds.

      Now let $X \in \mod R/\langle 1-e \rangle$ have socle $\soc(X)$ with $m$-dual $S$ with respect to $\add S_e$ and $m := m(\soc(X), S)$.
      We have $m < \infty$ because $\soc(X), S \in \add S_e$ and $S_e$ has the $m$-property.
      Then $m(S_e, S) \leq m$ because $S$ is an $m$-dual of $\soc(X)$.
      Further, $m(X'', S) \leq m$ for $X'' \in \mod R/\langle 1-e\rangle$ by \cref{eqn:Sedominates}.
      Therefore, the right term of the exact sequence 
      \includepdf{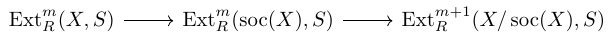}
      vanishes, so $m \leq m(X,S)$. 
      But then $m(X'', S) \leq m \leq m(X,S)$ for $X'' \in \mod R /\langle 1-e\rangle$ implies that $S$ is an $m$-dual of $X$ with respect to $\mod R/\langle 1-e\rangle$.
      This shows (\ref{enum:lemma1}) and (\ref{enum:lemma2}) follows since we could have picked $S = S_e$ as $m$-dual of $\soc(X)$ with respect to $\add S_e$.
      
      It remains to show that the restriction of $m(-,-)$ to $\mod R/\langle 1-e\rangle \times \mod R/\langle 1-e\rangle$ is a bounded above function.
      Let $m := m(X'',S_e) \leq \infty$ for some $X'' \in \mod R/\langle 1-e\rangle$.
      Then the composition series $0 = X'_0 \subset X'_1 \subset \dots \subset X'_{i} = X'$ of an arbitrary $X' \in \mod R/\langle 1-e\rangle$ gives rise to exact sequences
      \includepdf{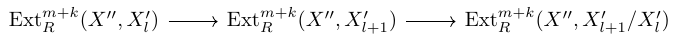}
      for $k \in \NN$ and $l \in \{0,\dots,i-1\}$, which imply $m(X'',X') \leq m$ by induction, similarly to the first part.
      In particular, $m(X'', X') \leq m(X'', S_e) \leq m(S_e, S_e) < \infty$, where we used \cref{eqn:Sedominates} for the second inequality, holds for $X',X'' \in \mod R/\langle 1-e\rangle$ and completes the proof that $\mod R /\langle 1-e\rangle$ has the $m$-property.
   \end{proof}

   \begin{theorem} \label{thm:generalization}
      Let $R$ be a semiperfect noetherian basic ring, $R/\langle 1-e\rangle$ be right artinian and $S_e$ have the $m$-property.
      If $\Tor_{k}^{\Gamma_e}(eR(1-e), (1-e)Re) = 0$ for all but finitely many $k \in \NN$ then
      \begin{enumerate} 
         \item there is a non-trivial $S \in \add S_e$ with $\Ext_R^k(S_e, S) = 0$ for $k > 0$ and\label{item:firstpart}
         \item if $Y_e$ is of uniformly graded Loewy length then (\ref{item:firstpart}) holds for $S = S_e$.\label{item:secondpart}
      \end{enumerate}
   \end{theorem}
   \begin{proof}
      Apply \Cref{cons:the_construction} to $\Omega_0 = eR$.
      We have $X_0 \neq 0$, because $eR$ is the projective cover of $S_e$.
      Notice, $0 = \Tor^{\Gamma_e}_k(eR(1-e), (1-e)Re) \cong \Tor^{\Gamma_e}_k(eR(1-e), (1-e)R)$ holds for all but finitely many $k \geq 1$ because $(1-e)R = \Gamma_e \oplus (1-e)Re$ as $\Gamma_e$-modules.
      Since $F_e(\Omega_0) = eR(1-e)$ there is a maximal $i \in \NN$ with $X_i \neq 0$ by \Cref{prop:homology}.

      Let $S$ be an $m$-dual for $X_i$ with respect to $\mod R/\langle 1-e\rangle$, where we choose $S := S_e$ if $Y_e$ is of uniformly graded Loewy length, and 
      \begin{equation} 
         m := m(X_i,S) = \sup \set{m(X', S)}{X' \in \mod R/\langle 1-e\rangle} \text{.} \label{equation:msup}
      \end{equation}
      We now apply $\Hom_R(-,S)$ to each row and column of \Cref{fig:infinite_tree} and look at the relevant diagrams.
      First we claim that the left and right term in the exact sequence
      \includepdf{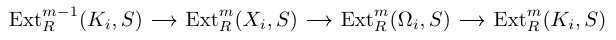}
      vanish.
      Indeed, since $i$ is maximal with $X_{i} \neq 0$ the minimal projective resolution of $K_i$ is given by $\cdots \to P_{i+2} \to P_{i+1} \to P_i \to K_i \to 0$.
      Hence, $\Ext_R^k(K_i, S) = 0$ for all $k \in \NN$ using \Cref{remark:determine} and $S \in \add S_e$ as well as $P_{k'} \in \add (1-e)R$ for $k' \in \NN$ by construction.
      So
      \begin{equation} \label{eqn:1stiso}
         \phi := \Ext_R^m(f_i, S) \colon \Ext_R^m(X_i, S) \to \Ext_R^m(\Omega_i, S)
      \end{equation} 
      is an isomorphism.
      If $i \geq 1$ the series of exact sequences
      \includepdf{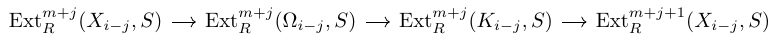}
      shows that there are isomorphisms 
      \begin{equation} \label{eqn:2ndiso}
         \kappa_j := \Ext_R^{m+j}(k_{i-j}, S) \colon \Ext_R^{m+j}(\Omega_{i-j}, S) \to \Ext_R^{m+j}(K_{i-j}, S)
      \end{equation} 
      for every $1 \leq j \leq i$, since $X_{i-j} \in \mod R/\langle 1-e\rangle$ and hence $\Ext_R^{m+k}(X_{i-j}, S) = 0$ for $k > 0$ by \cref{equation:msup}.
      Finally, if $i \geq 1$ the vertical short exact sequences induce exact sequences
      \includepdf{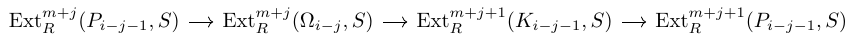}
      which give us isomorphisms 
      \begin{equation} \label{eqn:3rdiso}
         \partial_j \colon \Ext_R^{m+j}(\Omega_{i-j}, S) \to \Ext_R^{m+j+1}(K_{i-j-1}, S)
      \end{equation} 
      for $0 \leq j \leq i-1$ as $\Ext_R^k(P_{i}, S) = 0$ for $k \in \NN$ using \Cref{remark:determine} and $S \in \add S_e$ as well as $P_{i-j-1} \in \add(1-e)R$ by construction.
      
      Now, there is an isomorphism $\Ext_R^m(X_i, S) \to \Ext_R^{m}(\Omega_0,S)$ given by $\phi$ if $i=0$ and by the composite $\kappa_i^{-1} \partial_{i-1} \kappa_{i-1}^{-1} \cdots \partial_2 \kappa_{2}^{-1} \partial_1 \kappa_1^{-1} \partial_0 \phi$ of the isomorphisms from equations (\ref{eqn:1stiso}), (\ref{eqn:2ndiso}) and (\ref{eqn:3rdiso}) if $i \geq 1$. 
      Therefore, $m=i=0$ as $\Omega_0 = eR$ is projective.
      In particular, this shows $m(S_e, S) = m = 0$, which proves the theorem.
   \end{proof}
   The condition that $R/\langle 1-e\rangle$ is right artinian cannot be dropped easily as the ring of formal power series $R:= \kk\llbracket X\rrbracket$ shows.
Indeed, for $e = 1$ the simple $R$-module $\kk = S_{e}$ satisfies the assumptions of \Cref{thm:generalization} except for $R/\langle 0 \rangle$ being artinian, but ${\Ext_R^1(\kk, \kk) \neq0}$.

   \begin{corollary} \label{corollary:generalization}
      Let $R$ be a semiperfect noetherian basic ring, $e$ be a primitive idempotent and $R/\langle 1-e\rangle$ be artinian.
      If $\Tor_k^{\Gamma_e}(eR(1-e), (1-e)Re)$ and $\Ext_{R}^k(S_e,S_e)$ vanish for all but finitely many $k \in \NN$ then $S_e$ is self-orthogonal.
   \end{corollary}
   \begin{proof}
      $S_e$ has the $m$-property because $\Ext_R^{\smash{k}}(S_e,S_e)$ is non-zero for only finitely many $k \in \NN$ and because $S_e$ is simple and hence trivially an $m$-dual to itself.
      The result now follows from \Cref{thm:generalization}(\ref{item:firstpart}).
   \end{proof}

   We can give a proof of \Cref{conj:ingalls} for semiperfect noetherian rings now:
   \begin{proof}[Proof of \Cref{conj:ingalls}] \label{proof:ingalls}
      We have $\Ext_{\smash{R/\langle 1-e\rangle}}^1(S_e,S_e) = 0$ because $\mod R/\langle 1-e\rangle$ is a Serre subcategory of $\mod R$ by \Cref{lemma:serre} and because $\Ext^1_R(S_e,S_e) = 0$.
      This combined with $(R/\langle 1-e \rangle)/(\rad R/\langle 1-e \rangle) \cong S_e$ as $R/\langle 1-e\rangle$-modules shows that $R/\langle 1-e\rangle$ is a semisimple ring.
      In particular $R/\langle 1-e\rangle$ is a right artinian ring.
      As $\pdim_R S_e$ is finite, $\Ext^k_R(S_e,S_e)$ can be non-zero for only finitely many $k \in \NN$.
      Further $\pdim_{\Gamma_e} eR(1-e)$ being finite implies that $eR(1-e)$ has finite flat dimension as a $\Gamma_e$-module.
      In particular, $\Tor^{\smash{\Gamma_e}}_k(eR(1-e), (1-e)Re)$ is zero for all but finitely many $k \in \NN$.
      But then the result follows from \Cref{corollary:generalization}.
   \end{proof}

   Next, we want to make use of \Cref{thm:generalization} and therefore introduce the following:
   \begin{definition}
      Let $M \in \mod R$. 
      Then the \emph{$\Ext^{>0}$-quiver} $Q := (Q_0,Q_1)$ of $M$ has 
      \begin{enumerate}
         \item vertices $Q_0 := \{\,\text{(isomorphism classes of) indecomposable summands of } M\,\}$ and
         \item arrows $Q_1 := \set{M_1 \to M_2}{M_1,M_2 \in Q_0 \text{ and } m(M_1,M_2) > 0}$.
      \end{enumerate}
   \end{definition}

   Notice, this quiver is not weighted.
   There exists \emph{precisely one} arrow from $M_1$ to $M_2$ in $Q$ if $\Ext_R^{k}(M_1, M_2)$ does not vanish for some $k \geq 1$.
   Further, if $M'$ is a direct summand of $M$ then the $\Ext^{\smash{>0}}$-quiver of $M'$ is the full subquiver of the $\Ext^{\smash{>0}}$-quiver of $M$ having indecomposable summands of $M'$ as vertices.
   This is because the existence of an arrow between vertices of an $\Ext^{>0}$-quiver is solely determined by the homological relation of these vertices to each other.

   Recall that a vertex of a quiver is called \emph{source vertex} if there are no arrows ending in it.
   Dually a vertex is called \emph{sink vertex} if there is no arrow starting in it.
   \begin{remark} \label{remark:label}
      The vertices $Q_0$ of a quiver $Q = (Q_0, Q_1)$ admit a labeling $\{q_1, q_2, \dots, q_n\} = Q_0$ such that there is no arrow $q_i \to q_j$ for $i \leq j$ if and only if $Q$ does not contain an oriented cycle (here a loop is seen as an oriented cycle of length $1$).
   \end{remark}

   We also introduce the following intuitive notation:
   For an idempotent $e' \in R$ we call the idempotent subring $\Gamma_{e'}$ \emph{sandwiched} between $\Gamma_e$ and $R$ if $\Gamma_e \subset \Gamma_{e'} \subset R$.
   This is equivalent to $e'e = e e' = e'$.

   \begin{theorem} \label{thm:sandwiched_subrings}
      Let $R$ be a semiperfect noetherian basic ring and $R/\langle 1-e\rangle$ be right artinian.
      Suppose further that all idempotent subrings $\Gamma$ sandwiched between $\Gamma_e$ and $R$ have finite global dimension.
      Then there is a decomposition $S_e = S_1 \oplus \dots \oplus S_n$ into simple summands such that $\Ext_R^k(S_i, S_j) = 0$ for $1 \leq j \leq i \leq n$ and $k > 0$.
   \end{theorem}
   \begin{proof}
      Choose a decomposition $S_e = M_1 \oplus \cdots \oplus M_n$ into simple summands and let $Q$ be the $\Ext^{>0}$-quiver of $S_e$.
      Recall that $\add S_e$ is Krull--Schmidt by \cite[Theorem 12.6]{anderson_rings_1992}.
      Hence, $Q$ has precisely $n$ vertices.
      By \Cref{remark:label} the statement of the \namecref{thm:sandwiched_subrings} is equivalent to $Q$ not containing any oriented cycle, since the simple summands of $S_e$ are the vertices of $Q$ and non-trivial extensions between simple summands of $S_e$ belong to arrows in $Q$.
      We use induction on the number $n$ of simple summands in $S_e$ to show that $Q$ has no oriented cycles.
      For $n = 0$ the statement is trivial so let $n > 0$ and suppose the \namecref{thm:sandwiched_subrings} holds for all $n' < n$.

      Let $e' \in R$ be an idempotent so that $\Gamma_e \subsetneq \Gamma_{e'} \subset R$.
      Then the $\Ext^{>0}$-quiver of $S_{e'}$ is the proper full subquiver of $Q$ with the direct summands of $S_{e'}$ as vertices.
      Notice that every proper full subquiver of $Q$ arises this way as the $\Ext^1$-quiver of $S_{e'}$ for an idempotent $e' \in R$ with $\Gamma_e \subsetneq \Gamma_{e'} \subseteq R$.
      By induction hypothesis we can therefore assume that no proper full subquiver of $Q$ has any oriented cycles.

      Now suppose $Q$ had an oriented cycle.
      Every oriented cycle of $Q$ has to contain every vertex at least once because otherwise it would be contained in a proper full subquiver.
      Any arrow not contained in an oriented cycle of minimal length in $Q$ would give rise to an oriented cycle of smaller length using that there is at most one arrow between each pair of vertices in $Q$.
      So an oriented cycle of minimal length in $Q$ has to contain all vertices and all arrows of $Q$.
      This means that $Q$ is \emph{the} oriented cycle of length $n$.

      Now, $m(S_e,S_e)$ is finite because $R$ has finite global dimension.
      Next, there is an entry neither equal to $0$ nor to $-\infty$ in row $i$ and column $j$ of \Cref{tab:m-property} if and only if there is an arrow $M_i \to M_j$ in $Q$.
      As $Q$ is an oriented cycle there appears exactly one value in each row and each column which is neither $0$ nor $-\infty$, while all the other values are either $0$ or $-\infty$.
      This means that $S_e$ has the $m$-property by \Cref{remark:mproperty} as for each row the unique value outside $\{-\infty, 0\}$ is maximal within its column.
      However, then \Cref{thm:generalization} shows that there is a source vertex $S$ in $Q$.
      This is a contradiction because $Q$ is the oriented cycle of length $n$ and there are no source vertices in an oriented cycle.
      Hence, the induction is complete.
   \end{proof}

   Finally, recall the notion of directed rings, which are quasi-hereditary and hence of finite global dimension:

   \begin{definition} \label{def:triangular}
      An artinian ring $R$ is called \emph{directed} if there is a complete set of primitive idempotents $e_1, \dots, e_n \in R$ such that $e_i R e_j$ vanishes for all $1 \leq i < j \leq n$ and such that $e_i R e_i$ is a skew field for all $1 \leq i \leq n$. 
   \end{definition}

   \begin{corollary} \label{corollary:triangular}
      Let $R$ be a semiperfect noetherian basic ring of finite global dimension and $e \in R$ be an idempotent such that $R/\langle 1-e\rangle$ is right artinian.
      Then all idempotent subrings sandwiched between $\Gamma_e$ and $R$ have finite global dimension if and only if $Y_e$ is a directed artinian ring.
   \end{corollary}
   \begin{proof}
      If $Y_e$ is a directed artinian ring then all idempotent subrings of $Y_e$ are directed artinian rings too. 
      In particular, for any idempotent $e'$ with $\Gamma_e \subset \Gamma_{e'} \subset R$ the ring $Y_{e'}$ is artinian and of finite global dimension.
      It follows from \cite[Theorem 1.1]{ingalls_homological_2017} that $\Gamma_{e'}$ has finite global dimension.

      Conversely, $\gl R < \infty$ implies that $Y_e$ is artinian (see \cite[Section 4, page 312]{ingalls_homological_2017}).
      By \Cref{thm:sandwiched_subrings} there is a labeling $S_1 \oplus \cdots \oplus S_n = S_e$ of the simple summands of $S_e$ such that $\Ext^k_R(S_i, S_j) = 0$ for $1 \leq i \leq j \leq n$ and $k > 0$.
      Let $f_i \in \End_R (S_e)$ be the idempotent corresponding to $S_i$ for $1 \leq i \leq n$.
      Then $\Hom_{Y_e}(f_j Y_e, f_i Y_e) = f_i Y_e f_j \cong \Ext^\ast_R(S_j, S_i) = 0$ for $1 \leq i < j \leq n$ and $\Hom_{Y_e}(f_i Y_e, f_i Y_e) = f_i Y_e f_i \cong \Ext^\ast_R(S_i, S_i) = \Hom_R(S_i, S_i)$ is a skew field for $1 \leq i \leq n$.
      This shows that $Y_e$ is a directed artinian ring.
   \end{proof}

   \begin{example}
      Consider the quivers 
      \begin{align*} 
         Q  &:= \includeformula{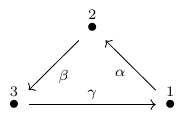}  & Q' &:= \includeformula{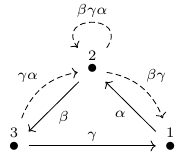}
      \end{align*}
      and the algebra $R := \kk Q/I$, where $I = \langle \gamma \alpha, \beta \gamma \rangle$. 
      Let $S := S_{e_1 + e_2 + e_3}$.
      An easy calculation shows that $R$ has global dimension $3$ and the $\Ext^{>0}$-quiver of $S$ has the shape of $Q'$. 
      Furthermore, $Y^{\op}_{e_1 + e_2 + e_3} \cong \kk Q / I^{\perp}$, where $I^{\perp} = \langle \alpha \beta\rangle$ and $\alpha$, $\beta$ and $\gamma$ are all of degree $1$.
      Notice, here $Y^{\op}_{e_1+e_2+e_3}$ has a $\kk$-basis given by the vertices and arrows of $Q'$.

      The algebras $\Gamma_{e_1} \cong \kk A_2$, $\Gamma_{e_3} \cong \kk A_2$ and $\Gamma_{e_1+e_3} \cong \kk$, where $A_2 = \includeformula[height=.9em]{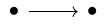}$, have finite global dimension.
      Therefore, $Y_{e_1 + e_3} \cong \kk A_2$ must indeed be directed artinian by \Cref{corollary:triangular}.
      On the other hand, the subalgebra $\Gamma_{e_2} \cong \kk(\includeformula[height=.9em]{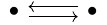})/\rad^2 \kk(\includeformula[height=.9em]{tikz/A2-tilde})$ has infinite global dimension, so $Y_{e_2} \cong \kk[X]/\langle X^2 \rangle$ must indeed not be directed artinian.

      Notice, $\Gamma_{e_1 + e_2} \cong \kk$ and $\Gamma_{e_2 + e_3} \cong \kk$ have finite global dimension even though $Y_{e_1+e_2}$ and $Y_{e_2+e_3}$ are \emph{not} directed.
      This is possible since \Cref{corollary:triangular} only implies the existence of a \emph{sandwiched} idempotent subalgebra of infinite global dimension, here $\Gamma_{e_2}$.
   \end{example}

   \begin{acknowledgment}
      This paper is based on the author's master's thesis. 
      Special thanks to Steffen König for his professional and linguistic support and advice throughout the writing of my master's thesis and the creation of this paper, as well as his corrections.
      Thanks to Jonas Dallendörfer for checking my master's thesis for linguistic mistakes.
      Thanks to the anonymous Referee for very useful suggestions.

      Also special thanks to Peter Jørgensen for professional support and advice and for his corrections.
      
      This work has been financially supported by the Aarhus University Research Foundation (grant no.\s\ AUFF-F-2020-7-16).
   \end{acknowledgment}

   \bibliographystyle{alpha}
   
\end{document}